%% file: instabilityAMS.tex
\documentclass{amsart}

\usepackage{amsthm}
\usepackage{thmtools}
\usepackage{amssymb}
\usepackage{amsmath}
\usepackage{microtype}
\usepackage[bb=esstix]{mathalpha}
\usepackage{mathtools}
\usepackage[hypertexnames=false]{hyperref}
\RequirePackage[capitalize]{cleveref}
\crefformat{equation}{(#2#1#3)}
\crefrangeformat{equation}{(#3#1#4) to~(#5#2#6)}
\crefmultiformat{equation}{(#2#1#3)}%
{ and~(#2#1#3)}{, (#2#1#3)}{ and~(#2#1#3)}

\input{mydefs}

\newtheorem{theorem}{Theorem}[section]
\newtheorem{lemma}[theorem]{Lemma}
\newtheorem{proposition}[theorem]{Proposition}

\theoremstyle{definition}
\newtheorem{definition}[theorem]{Definition}

\theoremstyle{remark}

\numberwithin{equation}{section}

\let\oldappendix\appendix
\renewcommand{\appendix}{%
\oldappendix%
\numberwithin{theorem}{subsection}%
\numberwithin{equation}{subsection}%
\renewcommand{\thesubsection}{\Alph{subsection}}%
\renewcommand{\theequation}{\thesubsection.\arabic{equation}}%
\renewcommand{\thelemma}{\thesubsection.\arabic{lemma}}%
\renewcommand{\theproposition}{\thesubsection.\arabic{proposition}}%
\renewcommand{\thedefinition}{\thesubsection.\arabic{definition}}%
}

\begin{document}

            \title[Exponential Instability of an Inverse Problem for the Wave Equation]{On Exponential Instability of an Inverse Problem for the Wave Equation}

    \author{Leonard Busch}
    \address{Korteweg-de Vries Institute for Mathematics, University of Amsterdam, Amsterdam, The Netherlands}
    \curraddr{Korteweg-de Vries Institute for Mathematics, 1098 XG Amsterdam, The Netherlands}
    \email{l.a.busch@uva.nl}
    %\thanks{This research was performed while the first author was visiting the University of Helsinki.}
    
    \author{Matti Lassas}
    \address{Department of Mathematics and Statistics, University of Helsinki, Helsinki, Finland}
    \email{matti.lassas@helsinki.fi}
    
    \author{Lauri Oksanen}
    \address{Department of Mathematics and Statistics, University of Helsinki, Helsinki, Finland}
    \email{lauri.oksanen@helsinki.fi}
    
    \author{Mikko Salo}
    \address{Department of Mathematics and Statistics, University of Jyväskylä, Finland}
    \email{mikko.j.salo@jyu.fi}


\date{June 30, 2025}


\begin{abstract}
        For a time-independent potential $q\in L^\infty$, consider the source-to-solution operator that maps a source $f$ to the solution $u=u(t,x)$ of $(\Box+q)u=f$ in Euclidean space with an obstacle, where we impose on $u$ vanishing Cauchy data at $t=0$ and vanishing Dirichlet data at the boundary of the obstacle. 
        We study the inverse problem of recovering the potential $q$ from this source-to-solution map restricted to some measurement domain. %
        By giving an example where measurements take place in some subset and the support of $q$ lies in the `shadow region' of the obstacle, we show that recovery of $q$ is exponentially unstable.
\end{abstract}

\maketitle

\section{Introduction}

Let $T>0$ be arbitrary and consider the strictly convex, compact obstacle $\mathcal{O} \coloneqq \overline{B_1(0)} \subset \RR^n$ with analytic boundary. Here and throughout, for any $r_0>0$ and $x_0\in \RR^n$, $B_{r_0}(x_0)$ denotes the open ball of radius $r_0$ about $x_0$ in $\RR^n$. Define $U \coloneqq B_{T+4}(0)\setminus \mathcal{O}$ (this choice being motivated later on) and introduce $X \coloneqq (-T,T)\times U$ and $X_+ \coloneqq (0,T)\times U$, where the definitions made so far remain fixed throughout this document. The wave operator in $X$ is denoted by $\Box = \partial_t^2 - \Delta_x$.

Let $q\in L^\infty(U)$. We consider the forward problem of finding a solution $u = u(t,x)$ to 
\begin{equation}\label{eq:forward}
	\Box u + q u = f\quad\text{in}\ \ X\,,\qquad u\vert_{t<0}=0\,,\quad u\vert_{(-T,T)\times\del U}=0%
\end{equation}
for any $f\in L^2(X_+)$. In fact, as shown in the appendix, \cref{lem:ex}, there is a continuous linear operator 
\[
    S_q \colon L^2(X_+) \to C([-T,T];H^1_0(U))\cap C^1([-T,T];L^2(U))
\]
so that $S_q(f)$ is the unique solution to \cref{eq:forward}. 

The aim of this note is to show that the inverse problem of determining $q$ from $S_q$ is exponentially unstable in certain partial data settings. This is achieved by appealing to the machinery in \cite{zbMATH07465842} and imposing support conditions on potentials $q$ that induce Gevrey smoothing of $S_q-S_0$. Throughout this document $e_1 = (1,0,\dots,0)\in \RR^n$ refers to the first standard basis vector, $r \in (0,1)$ is a fixed constant, and for sets $A,B$ we use the notation $A \Subset B$ to mean that the closure of $A$ is a compact subset of $B$. %
\begin{theorem}\label{thm:main}
	Let $r\in (0,1)$ be fixed, $\Sigma \Subset B_{r}(2e_1), \Xi \Subset B_{r}(-2e_1)$ be open, $T' \in (0,T)$, and define $\Omega \coloneqq (0,T')\times \Xi$. Let $\mu\in\RR, \delta>0$ be fixed so that $\mu+\delta > n/2$ and define $K \coloneqq \left\{q \in H^{\mu+\delta}(U)\colon \supp \,q\subset \bar \Sigma\,, \norm{q}_{H^{\mu+\delta}(U)}\leq 1\right\}$.

	If $\omega$ is a modulus of continuity so that 
	\begin{equation*}%
		\norm{q_1-q_2}_{H^\mu(U)} \leq \omega\left(\norm{\mathbb{1}_{\Omega}(S_{q_1}-S_{q_2})\circ \mathbb{1}_{\Omega}}_{L^2(\Omega)\to L^2(\Omega)}\right)\,,\qquad q_1,q_2\in K\,,
	\end{equation*}
	then $\omega(s) \gtrsim \abs{\log s}^{-\delta\frac{6n+7}{n}}$ for $s$ small.
\end{theorem}
That $S_q$ is well-defined for $q \in K$ follows from the Sobolev embedding and \cref{lem:ex}. For a graphical representation of the sets $\mathcal{O}, B_r(\pm 2e_1)$ and the notion that $B_r(2e_1)$ lies in the `shadow region' of $\mathcal{O}$ from the perspective of $B_r(-2e_1)$, we point to \cref{fig:rays} below. 

The set $U$ was chosen so that for sources $f\in L^2(X_+)$ with $\supp \,f\subset [0,T]\times \overline{B_3(0)}$, by the finite speed of propagation, $S_q(f)$ vanishes identically near $(-T,T)\times \del B_{T+4}(0)$, the `far' boundary of $X$. Thus, the extension of $S_q(f)$ by $0$ outside of $U$ solves%
\[
    (\Box+q)S_q(f) = f\ \ \text{in}\ \ (-T,T)\times \RR^n\setminus \mathcal{O}\,,\quad S_q(f)\vert_{t<0}=0\,,\quad S_q(f)\vert_{(-T,T)\times\del\mathcal{O}}=0\,,
\]
when $\supp \,f\subset [0,T]\times \overline{B_3(0)}$, where we recall that $T>0$ was chosen freely. This clarifies that the only significant geometric feature of the space $X$ is the subset of its boundary $(-T,T)\times\del\mathcal{O}$. Because $B_r(\pm 2e_1)\subset B_{3}(0)$, considering only sources $f$ satisfying $\supp \,f\subset [0,T]\times \overline{B_3(0)}$ is no restriction in the context of \cref{thm:main}.

\cref{thm:main} will come as a consequence of the qualitative statement of the propagation of singularities (\cref{cor:gev}) together with facts from functional analysis to get a quantitative statement. Briefly: we have access to a complete description of the propagation of smooth singularities via \cite[Thm.~24.5.3]{zbMATH05129478} if $q\in C^\infty$. Furthermore, due to \cite{zbMATH03973398} (in particular \cite{zbMATH03891823}, and see also \cite[Thm.~2.1]{zbMATH04017953}), together with \cite[Thm.~7.3]{zbMATH03359011}, we have a complete description of the Gevrey-$\sigma$ singularities for $\sigma \geq 1$ if $q \in C^\omega$. We reduce to the case $q=0$ by considering $S_q-S_0$ so that these propagation results become applicable. The existence of the obstacle $\mathcal{O}$ prevents (Gevrey-$3$-)singularities of sources supported in $(0,T)\times \Sigma$ from propagating and giving rise to singularities of the solution in $\Omega$. The machinery of \cite{zbMATH07465842} will turn this smoothing behavior into an instability statement.%

The significance of \cref{thm:main} lies in the fact that the Boundary Control method, pioneered by M. Belishev in \cite{MR924687} and extended to Riemannian manifolds in \cite{zbMATH00147418}, see also \cite{MR1889089}, implies that one can indeed recover $q$ uniquely from the knowledge of $S_q$ (on $\Omega$) for $T<\infty$ sufficiently large: in the setting of a closed manifold and with $q \in C^\infty$, this is shown in \cite{saksala2025inverseproblemsymmetrichyperbolic} (see Remark 1.2 therein), whereas for $q\in L^\infty$, in Euclidean space without the presence of an obstacle see \cite[Thm.~1.1]{filippas2025stabilityinverseproblemwaves}. In the setting without an obstacle, \cite[Thm.~1.2]{filippas2025stabilityinverseproblemwaves} proves log-log stability of the solution operator $S_q$. For emphasis we point out that in the situation of \cref{thm:main}, despite the fact that measurements take place in a bounded domain for a large period of time, recovery is shown to be exponentially unstable.

For a different choice of measurement set $\Omega$, recovery of $q$ from $S_q$ can be proved by reduction to the closely related inverse problem using the Dirichlet-to-Neumann (DN) map as data, which we briefly elaborate on here (see also \cite{zbMATH02094176} for a discussion of various equivalent inverse problems). 
For $W\subset U$ open, $q\in L^\infty(W)$ and $H^1_+((0,T)\times\del W)\coloneqq \{\varphi \in H^1((0,T)\times \del W)\colon \varphi(0,\cdot) = 0\}$ one considers the operator $\Lambda_q\colon H^1_+((0,T)\times\del W)\to L^2((0,T)\times \del W)$ defined via $\Lambda_q \colon \varphi \mapsto \del_\nu v\vert_{(0,T)\times \del W}$, where $\nu$ is the outward pointing unit normal at $\del W$ and $v\in C([0,T];H^1(W))\cap C^1([0,T];L^2(W))$ is the unique solution of 
\[
	\Box v + q v = 0\quad\text{in}\ \ (0,T)\times W\,,\qquad v\vert_{t=0}=0=\del_t v\vert_{t=0}\,,\quad v\vert_{(0,T)\times\del W}=\varphi\,.
\]
That $\Lambda_q$ is well-defined is guaranteed, for example, by \cite[Lem.~1.2]{zbMATH06011852}. One then asks whether the knowledge of the map $\Lambda_q$ uniquely determines $q$, which is answered positively in \cite{zbMATH04092031} if $T>\mathrm{diam}\,W$. Here we note that in the case of data on the full boundary as we have here, the Dirichlet-to-Neumann and Neumann-to-Dirichlet data formulations are equivalent.

If we denote by $\eta\colon \{\varphi \in C^\infty((0,T)\times\del W)\colon \supp \,\varphi\Subset \{t\geq 0\}\} \to C^\infty((0,T)\times U)$ an extension operator (that is $\eta(\varphi)\vert_{(0,T)\times \del W} = \varphi$ for all $\varphi$), then a calculation shows that for all $\varphi \in C^\infty((0,T)\times\del W)$ with $\supp \,\varphi\Subset \{t\geq 0\}$,
\[
	\Lambda_q(\varphi) = \del_\nu(\eta(\varphi)-S_q((\Box+q)\eta(\varphi)))\vert_{(0,T)\times\del W}\,.
\]
We conclude (using a density argument) that if $W\subset U$ is open, $\supp \,q \subset W$, and the measurement set $\Omega$ contains a neighborhood of the boundary $(-T,T)\times\del W$, the inverse problem for the source-to-solution map with data on $\Omega$ can be reduced to that with the DN map as data, which admits a large body of literature. In particular, for the DN-map formulation with $W=B_R(0) \setminus \mathcal{O}$ with $R$ fixed (so one has measurements on both the outer and inner boundaries of $W$), \cite{zbMATH04212656} shows that recovery of $q$ is H\"older stable if $T$ is sufficiently large, see also \cite{zbMATH06011852}. We refer also to \cite{zbMATH06291149,zbMATH07138389,zbMATH07684958} and the references therein for further stability results for the case of DN-map data.

Another related case is where one has access to measurements on the full outer boundary $\partial B_{R}(0)$, but there is a vanishing Dirichlet boundary condition on $\partial \mathcal{O}$. In this case, if $T$ is sufficiently large, the DN-map determines H\"older stably the integrals of $q$ over line segments not touching $\mathcal{O}$ via geometrical optics solutions \cite{zbMATH06881246}. From these integrals one can determine $q$ H\"older stably when $n \geq 3$ (e.g.\ by inverting the X-ray transform on two-dimensional slices), whereas for $n=2$ one would also need to use broken lines reflecting on $\partial \mathcal{O}$ \cite{zbMATH02152806}. \cref{thm:main} corresponds to a setting where one has measurements in a smaller region, thus leading to exponential instability of the inverse problem.

We close the introduction by mentioning that in the setting of the Calder\'on problem, it follows from \cite{zbMATH03998383} and \cite{zbMATH01691016} that recovery of a potential for the Schr\"odinger operator using the Dirichlet-to-Neumann map is logarithmically stable, and that this stability is optimal, see also the discussion in \cite[\S~1.1]{zbMATH07465842}. In the absence of a potential $q$, the Gel'fand problem of recovering properties of a manifold (such as its metric) from spectral data of the Laplacian is discussed in \cite{zbMATH02132012,zbMATH07511975,zbMATH08023890}, the first of which establishes an abstract stability result and the latter two show log-log stability.

\subsection*{Acknowledgments}
This research was performed while L.B.\ was visiting the University of Helsinki. M.L.\ was partially supported by the Advanced Grant project 101097198 of the European Research Council, Centre of Excellence of Research Council of Finland (grant 336786) and the FAME flagship of the Research Council of Finland (grant 359186). L.O. was supported by the European Research Council of the European Union, grant 101086697 (LoCal),
and the Research Council of Finland, grants 347715, 353096 (Centre of Excellence of Inverse Modelling and Imaging) and 359182 (Flagship of Advanced Mathematics for Sensing Imaging and Modelling). M.S.\ was supported by the Research Council of Finland (Centre of Excellence in Inverse Modelling and Imaging and FAME Flagship, grants 353091 and 359208). Views and opinions expressed are those of the authors only and do not necessarily reflect those of the European Union or the other funding organizations. 

\section{The Proof}

Let $p(t,x,\tau,\xi) = -\tau^2 + |\xi|^2$ be the principal symbol of $\Box$, which will remain fixed throughout. We begin with two purely geometric statements that can be summarized informally as: generalized bicharacteristic arcs of $p$ with a point in the cotangent bundle $T^\ast ((0,T)\times B_{r}(-2e_1))$ cannot have `originated' from a neighborhood of $T^\ast(\{0\}\times B_{r}(2e_1))$. The following two statements are immediate when one draws a picture (see \cref{fig:rays}), but we give the full details nevertheless, first fixing some terminology. If $H$ is a hyperplane in $\RR^n$ defined by $H = \{x \colon (x-x^\circ)\cdot x^\circ = 0\}$ where $x^\circ \in \RR^n$ is some point, we say that a point $x\in \RR^n$ lies above (resp. below) $H$ if $(x-x^\circ)\cdot x^\circ > 0$ (resp. $<0$). 
\begin{lemma}\label{lem:hyp}
	Let $t^\circ>0$ and $\gamma \colon [0,t^\circ] \to \bar U$ be a line segment with $\gamma(0) \in B_{r}(-2e_1)$, $\gamma(t^\circ)\in \del \mathcal{O}$, and $H$ the tangential hyperplane of $\del\mathcal{O}$ at the point $\gamma(t^\circ)$, oriented so that the origin lies below $H$. The set $B_{r}(2e_1)$ lies below $H$.
\end{lemma}
\begin{proof}
	Denoting $x^\circ = \gamma(t^\circ)$, the normal of $\del\mathcal{O}$ at $x^\circ$ is given by $x^\circ$, so that $H$ is defined as $H = \{x\colon (x-x^\circ)\cdot x^\circ = 0\}$.%

	Because any line segment below $H$ intersecting $H$ at $x^\circ$ must pass through $\mathcal{O}$ and $\gamma$ describes a line segment in $\bar U$, all of $\gamma(t)$ and thus $\gamma(0)$ must lie above (or on) $H$. %
	Let $\gamma(0) = x_0 \in B_{r}(-2e_1)$, which can be written as $x_0 = -2e_1+y_0$ for some $\abs{y_0}<r$. Because $x_0$ lies above (or on) $H$, we conclude that
	\[
		0 \leq (x_0-x^\circ)\cdot x^\circ = -2x^\circ_1 -1 + y_0\cdot x^\circ \leq -2x^\circ_1-1+r\,,
	\]
	where we used that $\abs{x^\circ}=1$ and the Cauchy-Schwarz inequality, and $x^\circ_1$ is the first component of $x^\circ$. We conclude that $x^\circ_1 < 0$. 

	Similarly, any point $x\in B_{r}(2e_1)$ can be written as $x = 2e_1 + y$ with $\abs{y}< r$ so that
	\[
		(x-x^\circ)\cdot x^\circ < 2x_1^\circ + r-1 < r-1\,,
	\]
	where we used $x^\circ_1 < 0$. This shows that $B_{r}(2e_1)$ lies below $H$, completing the proof.
\end{proof}

\begin{lemma}\label{lem:bich}
	For every $\nu^\ast = (t^\ast,x^\ast,\tau^\ast,\xi^\ast)\in T^\ast((0,T)\times B_{r}(-2e_1))\sem$ with $\abs{\tau^\ast} = 1$ and $\nu^\ast \in p^{-1}(0)$, the forward generalized bicharacteristic arc $\gamma$ (according to \cite[Def.~24.3.7]{zbMATH05129478} or \cite{zbMATH03574269}) on $\RR\times U$ with initial condition $\gamma(t^\ast)=\nu^\ast$ satisfies $\gamma(t) \in T^\ast (\{t\}\times (U\setminus B_{r}(2e_1)))\sem$ for $t\in [0,t^\ast]$. Here, \emph{forward} means that the projection of $\gamma$ onto the $t$ component is an increasing function. %
\end{lemma}
\begin{proof}
	The following general remark will be used throughout the rest of the proof. Because $\gamma$ travels with unit speed forward in time ($\abs{\pi_\tau \gamma(t)}=\abs{\tau^\ast}=1, t\in [0,t^\ast]$) and the `far' boundary $\RR\times \del B_{T+4}(0)$ was chosen far away, and $\pi_x\gamma(t^\ast)\in B_{r}(-2e_1)$, we have $\pi_x\gamma(t) \not\in \del B_{T+4}(0)$ for all $t\in [0,t^\ast]$. We make a case distinction, see also \cref{fig:rays} for a pictorial view of the behavior of the ray $\pi_x\gamma$.
\begin{figure}
	\centering
	\begin{tikzpicture}
		\node[circle,draw,thick, minimum size=50pt,inner sep=0pt, outer sep=0pt] at (0,0) {};
		\node[circle,draw,thick, minimum size=37.5pt,inner sep=0pt, outer sep=0pt] at (-100pt,0) {};
		\node[circle,draw,thick, minimum size=37.5pt,inner sep=0pt, outer sep=0pt] at (100pt,0) {};
		\node at (-110pt,27pt) {$B_{r}(-2e_1)$};
		\node at (110pt,27pt) {$B_{r}(2e_1)$};
        \node at (-56pt,38pt) {$H$};
		\node at (15pt,-28pt) {$\mathcal{O}$};
		\draw[line width={0.7pt},dotted] (-110pt,10pt) -- (100pt,-15pt);
		\draw[line width={0.7pt}] (-90pt,-10pt) -- (70pt,57pt);
		\draw[line width={0.7pt}] (-95pt,5pt) -- (220:25pt);
		\draw[line width={0.7pt}] (220:25pt) -- (249:42pt);
		\draw[line width={0.6pt},dashed] (-70pt,40pt) -- (3pt,-41pt);
	\end{tikzpicture}
	\caption{A pictorial representation of $B_{r}(\pm 2e_1), \mathcal{O}$ in $U$ considered the ambient space with $\del B_{T+4}\subset \del U$ outside of frame. The following discussion is in the context of the proof of \cref{lem:bich}. The dotted line represents case 1: a line segment with endpoints in $B_{r}(\pm2e_1)$ must pass through the obstacle $\mathcal{O}$ (in particular intersecting $\del\mathcal{O}$); a ray cannot follow this path. The continuous lines are the two possible cases 2a, 2b: a ray starting in $B_{r}(-2e_1)$ intersects $\del\mathcal{O}$ either glancingly and continues as the same line segment, staying away from $B_{r}(2e_1)$. Or, the ray intersects $\del\mathcal{O}$ transversally, being reflected according to Snell's law off the tangential plane $H$ at the intersection point, represented here by the dashed line. In the latter case one sees that the ray always stays on one side of $H$, whereas $B_{r}(2e_1)$ lies on the other side of $H$.}
	\label{fig:rays}
\end{figure}
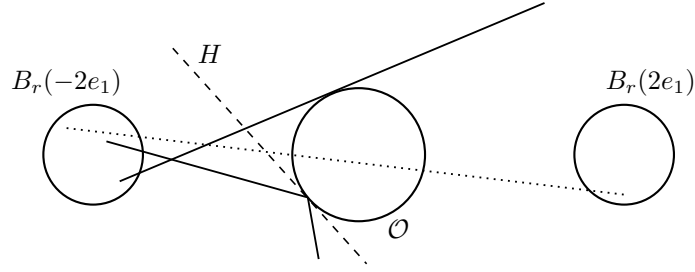
	
	Case 1: the generalized bicharacteristic arc defined by $\gamma(t)$, when projected onto the $x$-component, does not intersect $\del \mathcal{O}$ while $t\in [0,t^\ast]$. This means that $\pi_x\gamma(t), t\in [0,t^\ast]$ describes a line segment with $\pi_x\gamma(0)\in B_{r}(-2e_1)$, and if it were to intersect $B_{r}(2e_1)$, some point $\pi_x\gamma(t)$ must lie on $\del\mathcal{O}$, a contradiction. We thus find that $\pi_x\gamma(t) \not\in B_{r}(2e_1)$ for all $t\in[0,t^\ast]$, which completes the consideration of this case.

	Case 2: there is $t^\circ \in (0,t^\ast)$ so that %
	$\pi_{(t,x)} \gamma(t^\circ)\in (0,T)\times \del \mathcal{O}$. 
	We remark here that the normal of $\RR\times \del \mathcal{O}$ at any $(t,x)\in (0,T)\times\del\mathcal{O}$ is given by $(0,x)$. 

	Case 2a: the generalized bicharacteristic arc $\gamma$ intersects the boundary $(0,T)\times\del \mathcal{O}$ tangentially (glancingly) when $t=t^\circ$ (which is to say that $\lim_{t\to^+ t^\circ}\pi_{(\tau,\xi)}\gamma(t)\perp (0,\pi_{x}\gamma(t^\circ))$). Due to the fact that $\mathcal{O}$ is strictly convex and the remarks before \cite[Def.~24.3.2]{zbMATH05129478}, any glancing intersection of $\gamma$ with the boundary $(0,T)\times\del\mathcal{O}$ must be diffractive. 
	By the definition of a generalized bicharacteristic arc, $\gamma$ is unperturbed by diffractive intersections and continues as the same line segment. This implies that $\pi_x\gamma(t), t\in [0,t^\ast]$ intersects $\del U$ only at $t=t^\circ$, so that $\pi_x\gamma \colon [0,t^\ast]\to \bar U$ describes a line segment, which, as in case 1 implies that $\pi_x\gamma(t)\not\in B_{r}(2e_1)$ for all $t\in [0,t^\ast]$.

	Case 2b: the generalized bicharacterisitc arc $\gamma$ intersects the boundary $(0,T)\times \del \mathcal{O}$ transversally at $t=t^\circ$. We are thus in the case of an intersection in the hyperbolic region. Let $\gamma(t_\pm^\circ)\coloneqq \lim_{t\to^\pm t^\circ} \gamma(t) = ((t^\circ,x^\circ),\sigma_\pm)$ and let $H = \{y\in X\colon (y-(0,x^\circ))\cdot (0,x^\circ)\}$ be the tangential hypersurface of $(0,T)\times\del\mathcal{O}$ at $(0,x^\circ)$.  

	We decompose $\sigma_\pm$ into tangential and normal components with respect to $H$: $\sigma_\pm = \sigma_\pm^t + \sigma_\pm^n$, where $\sigma_\pm^n \in \RR (0,x^\circ)$ and $\sigma_+^n \neq 0$ by assumption of transversal intersection. In fact, because hyperbolic intersections are isolated (and there cannot have been a glancing intersection) we must have $\sigma_+^n\cdot (0,x^\circ)>0$ since $\pi_{(t,x)}\gamma(t)\in X$ for $t\in(t^\circ,t^\circ+\eps)$ for some $\eps>0$ (ie. $\gamma$ `comes from inside the set $X$'). By the definition of a generalized bicharacteristic arc, we must have $\sigma^+_t = \sigma^-_t$, and because $\gamma(t_\pm^\circ) \in p^{-1}(0)$, we find that $\sigma_-^n = -\sigma_+^n$ so that $\sigma_-^n \cdot (0,x^\circ)<0$. 

	We conclude that $\pi_{(t,x)}\gamma(t)$ lies above the hypersurface $H$ for $t$ near but unequal to $t^\circ$. Further (using that $(0,T)\times \mathcal{O}$ lies below $H$ aside from at $\RR\times \{x^\circ\}$), we see that $\pi_x \gamma(t)$ does not intersect $\del U$ for $t\in [0,t^\ast]\setminus t^\circ$, so that %
	$\pi_{(t,x)}\gamma$ always lies above the hypersurface $H$ for $t\in [0,t^\ast]\setminus t^\circ$.
	
	However, $(0,T)\times B_{r}(2e_1)$ lies below the hypersurface $H$ by \cref{lem:hyp}, completing the proof.
\end{proof}

Recall the definition of Gevrey spaces.
\begin{definition}\label{def:gev}
	For $d\in \mathbb{N}, \sigma \in [1,\infty)$ and $W\subset \RR^d$ open, we let $G^\sigma(W)$ be the set of all $f\in C^\infty(W;\mathbb{C})$ so that for every compact $K\subset W$ there is some $C>0$ so that for all multi-indices $\alpha$ one has
	\[
		\max_K \abs{\del^\alpha f} \leq C^{1+\abs{\alpha}}\abs{\alpha}^{\sigma\abs{\alpha}}\,.
	\]
\end{definition}
We remark that $G^1(W)$ is the set of analytic functions on $W$. The definition of the Gevrey-$\sigma$ wavefront set, an analogue of the smooth wavefront set in Gevrey-$\sigma$-regularity, can be found in \cite[\S~8.4]{hoermander1} and \cite{doi:10.1142/1550}; we shall not directly require this notion.

Let ${Y}\coloneqq (-T,T)\times B_{r}(-2e_1) \subset X$. We exploit our knowledge of the propagation of singularities together with the geometric statement \cref{lem:bich} to give
\begin{proposition}\label{cor:gev}
	For any open $\Sigma\Subset B_{r}(2e_1)$ and any $g\in L^2(U)$ with $\supp \,g\subset [0,T]\times \bar\Sigma$, if $w\in H^1(X)$ satisfies
	\[
		\Box w = g\quad\text{in}\ \ X\,,\qquad w\vert_{t<0}=0\,,\quad w\vert_{(-T,T)\times\del U}=0\,,%
	\]
	then $w\vert_{{Y}}\in G^3({Y})$.
\end{proposition}	
\begin{proof}%
	From \cite[Thm.~24.1.4]{zbMATH05129478}, the fact that $\Sigma$ is away from $\del U$, and a compactness argument we know that there is some $\eps>0$ so that $w = 0$ in $X\cap((\{0\}\times\del U)+B_\eps(0))$ (Minkowski sum). Combined with an application of \cite[Thm.~23.2.7]{zbMATH05129478} together with a compactness argument over $U\setminus ((\del U+B_\eps(0))\cup B_{r}(2e_1))$ implies that %
	\begin{equation}\label{eq:0near0}
		\exists\eps>0 \colon w\vert_{X\cap((\{0\}\times U\setminus B_{r}(2e_1))+B_\eps(0))} = 0\,.
	\end{equation}
	In particular, we have already shown that for the same $\eps>0$ in \cref{eq:0near0}, we have $0 = w\vert_{(-T,\eps)\times B_{r}(-2e_1)} \in G^3((-T,\eps)\times B_{r}(-2e_1)) = G^3(Y \cap \{t< \eps\})$.

	Now assume that $\nu^\ast=(t^\ast,x^\ast,\tau^\ast,\xi^\ast)\in \mathrm{WF}(w)\cap T^\ast({Y}\cap \{t>0\})$ where we may assume that $\abs{\tau^\ast}=1$ by renormalizing and we know from \cite[Thm.~8.3.1]{hoermander1} (and $\supp  \,g\subset [0,T]\times\bar \Sigma$) that $\nu^\ast \in p^{-1}(0)$. By \cref{lem:bich}, the forward generalized bicharacteristic arc $\gamma$ from \cref{lem:bich} with $\gamma(t^\ast)=\nu^\ast$ satisfies $\gamma(t) \in T^\ast(\{t\}\times (U\setminus B_{r}(2e_1)))\sem$ for all $t\in [0,t^\ast]$. 

	We note that the support condition we assume on $g$ implies that $g\in\mathcal{N}(\bar X)$ (defined in \cite[Def.~18.3.30]{zbMATH05129478}), and according to \cite[Cor.~18.3.31]{zbMATH05129478}, we may assume $w \in \mathcal{N}(\bar X)$. Note further that the Hamiltonian $H_p$, is never radial because $\del_{(\tau,\xi)} p \neq 0$ for $(\tau,\xi) \neq 0$ (see e.g. \cite[Rem.~2.1]{MR4691487}).
 
 	Therefore, from an application of \cite[Thm.~24.5.3]{zbMATH05129478} (or combination of \cite[Thm.~23.2.9,~Thm.~24.2.1,~Thm.~24.4.1]{zbMATH05129478}), we find that $\gamma(0)\in \mathrm{WF}(w)$ (because $\gamma(t^\ast)\in \mathrm{WF}(w)$, and $\supp \,g \subset [0,T]\times \bar \Sigma$, $\Sigma\Subset B_{r}(2e_1)$). However, \cref{eq:0near0} and the fact that $\pi_{(t,x)}\gamma(0) \in \{0\}\times U\setminus B_{r}(2e_1)$ then leads to the desired contradiction so that 
	$w\vert_{{Y}} \in C^\infty({Y})$.

	Further, because \cite[Thm.~1.4]{zbMATH03973398} states that Gevrey-$3$ singularities for $w$ propagate precisely the same as smooth singularities near the analytic boundary $(-T,T)\times \del U\subset \del X$ (since $g$ vanishes there), replacing the application of \cite[Thm.~8.3.1]{hoermander1} with \cite[Thm.~8.6.1]{hoermander1}, and using the propagation of Gevrey-$3$ singularities in the interior (\cite[Thm.~7.3]{zbMATH03359011}), we have completed the proof.
\end{proof}
The reason we do not hope to achieve better that Gevrey-$3$ regularity for $w$ above comes from the fact that Gevrey-$\sigma$ singularities for $\sigma \in [1,3)$ do not behave like the smooth ones at the boundary $\del \mathcal{O}$; instead, they behave like `analytic rays' (see \cite{zbMATH03973398,zbMATH03717772}). According to \cite{zbMATH03717772} and \cite[Thm.~0.5]{zbMATH03830467}, analytic rays can and do bend around the boundary of the obstacle and penetrate into the `shadow region', in which the non-regular $q$ is supported (see also \cite{zbMATH03660049}).  

Since we will appeal to \cite[Thm.~4.2(b)]{zbMATH07465842} in order to prove \cref{thm:main}, we will have to replace the domain ${Y}$ by some closed manifold. We will also have to show mapping properties of $S_q$ into some Banach space rather than $G^3$. We thus recall the definition of Gevrey spaces on closed manifolds from \cite[\S~2.6]{zbMATH07465842} and their decomposition into a union of Banach spaces.
\begin{definition}
Let $d\in\mathbb{N}, \sigma \in [1,\infty)$. For a closed smooth manifold $(M,g)$ of dimension $d$ and any $\rho>0$, for $(\varphi_j)_{j\in\mathbb{N}}\subset L^2(M)$ an orthonormal basis of eigenfunctions for $-\Delta_g$, we introduce the subspace $A^{\sigma,\rho}(M)$ of $L^2(M)$ defined as those $u\in L^2(M)$ so that
	\[
		\norm{u}_{A^{\sigma,\rho}(M)} = \left(\sum_{j=0}^\infty e^{2\rho j^{\frac{1}{d\sigma}}}\abs{\langle u,\varphi_j\rangle}^2\right)^{1/2} < \infty\,.
	\]
	Furthermore, if $(M,g)$ is analytic, we let $G^\sigma(M)$ be the set of all $u\in C^\infty(M;\mathbb{C})$ so that for all $k\in \mathbb{N}$ and some $C>0$,
	\[
		\norm{\nabla^k u}_{L^\infty(M)} \leq C^{k+1}k^{\sigma k}\,,
	\]
	with the convention that $0^0=1$.
\end{definition}
If $(M,g)$ is a closed analytic manifold and $u\in G^{\sigma}(M)$ for some $\sigma \in [1,\infty)$, then in each coordinate chart, the function $u$ is of class $G^\sigma$ according to \cref{def:gev}, see \cite{zbMATH07264015}. 

The importance of the spaces $A^{\sigma,\rho}$ lies in the fact that $\bigcup_{\rho>0} A^{\sigma,\rho}(M) = G^\sigma(M)$ for closed analytic manifolds $(M,g)$, see \cite[\S~2.6,~\S~B]{zbMATH07465842}. There it is also shown that each $A^{\sigma,\rho}(M)$ is a Banach space when $(M,g)$ is a closed smooth manifold.

We shall only need a select amount of properties of the spaces $A^{\sigma,\rho}$, which we state for the readers' convenience.
\begin{lemma}[{\cite[\S~2.6,~Lem.~B.1]{zbMATH07465842}}]\label{lem:inA}
	Let $\sigma \in [1,\infty)$ and $(M,g)$ be a closed analytic manifold. If $u\in G^{\sigma}(M)$, there is some $\rho>0$ so that $u\in A^{\sigma,\rho}(M)$. Furthermore, if for some $C,R>0$, $u\in L^2(M)$ satisfies 
	\begin{equation}\label{eq:inA}
		\norm{(-\Delta_g)^tu}_{L^2(M)} \leq CR^{2t}(2t)^{2t\sigma}\
	\end{equation}
	for all $t\in \mathbb{N}$, then for some $\rho_0>0$ and any $\rho \leq \rho_0$, we have $u \in A^{\sigma,\rho}(M)$. 
\end{lemma}
\begin{proof}
	The first statement is a consequence of \cite[Lem.~B.1]{zbMATH07465842}. For the second, let $u\in L^2(M)$ satisfy \cref{eq:inA}. Once we have shown that $u$ satisfies \cref{eq:inA} for all $t\in [0,\infty)$, not just the integers, the proof is completed by \cite[Lem.~B.1]{zbMATH07465842}. The proof of the sufficiency statement in \cite[Lem.~B.1(a)]{zbMATH07465842} only requires \cref{eq:inA} for integers $t$, which we use to conclude that $u\in G^\sigma(M)$, so that the necessity statement of \cite[Lem.~B.1(a)]{zbMATH07465842} gives \cref{eq:inA} for possibly different $C,R>0$ and all $t\in [0,\infty)$, which completes the proof.
\end{proof}

Recall that we defined ${Y} = (-T,T)\times B_{r}(-2e_1)$. We set up a functional analytic result that will allow us to replace ${Y}$ by some closed manifold $N\supset {Y}$. It will also later be used to show that $S_q-S_0$ maps into $A^{3,\rho}$ for some $\rho>0$. We will rely on \cite{zbMATH03800337}, where for any $\sigma > 1$, our notation $G^\sigma$ corresponds to $\mathcal{E}^{\{M_p\}}$ in that paper, where $M_p$ is the sequence $M_0=1, M_p = p!^\sigma\,, p \in\mathbb{N}$.
\begin{proposition}\label{prop:kernel}
	Let $\sigma>1$ and $N$ be a closed analytic $(n+1)$-dimensional manifold with $N\Supset {Y}$. Let $F\colon L^2(X_+) \to G^\sigma({Y})$ be linear with $F \colon L^2(X_+)\to L^2({Y})$ continuous.

	For every $\chi \in G^\sigma(N)\cap C_c^\infty({Y})$ there is some $\rho>0$ so that $\chi F \colon L^2(X_+) \to A^{\sigma,\rho}(N)$ is continuous. %
\end{proposition}
\begin{proof}
	Because $F \colon L^2(X_+)\to L^2({Y})$ is linear and continuous, it has closed graph. Because $F \colon L^2(X_+)\to G^{\sigma}({Y})\subset H^{(n+1)/2+1}({Y})$ and the following inclusion map $\iota \colon H^{(n+1)/2+1}({Y})\to L^2({Y})$ is continuous, the closed graph theorem implies that $F \colon L^2(X_+)\to H^{(n+1)/2+1}({Y})$ is continuous and the dual map
	\[
		F^\ast \colon H^{-(n+1)/2-1}({Y}) \to L^2(X_+)'\,,%
	\]
	is well-defined. 
	In particular, we may define
	\[
		H \colon {Y} \to L^2(X_+)'\,, \quad x \mapsto \chi(x) F^\ast(\delta_x)\,,%
	\]
	and for all $u\in L^2(X_+)$, applying \cref{lem:inA} to $\chi$,
	\[
		\langle H(x), u\rangle = \chi(x)(Fu)(x) \in G^\sigma({Y})\,\quad\text{as a function of}\ x\,.
	\]
	Thus, \cite[Thm.~3.10]{zbMATH03800337} implies that in fact $H \in G^\sigma({Y},L^2(X_+)')$, %
	where the meaning of this space is explained in \cite[Def.~3.9]{zbMATH03800337}, which in our case reduces to: %
 	for every compact $K\subset {Y}$ there are constants $C,R>0$ so that for all multi-indices $\alpha$,
	\[
		\sup_{x\in K} \norm{\del^\alpha_x H(x)}_{L^2(X_+)'} \leq CR^{\abs{\alpha}}\alpha!^\sigma\,.
	\]
	(We point also to the remarks near \cite[Thm.~27.1,\S~40]{zbMATH05162168} for an explanation as to what differentiation of a topological-vector-space-valued function means.)
	In particular, since $\supp \,\chi \eqqcolon K\subset {Y}\subset N$ is compact, we find that 
	\begin{align*}%
		\sup_{\norm{u}_{L^2(X_+)}=1}\norm{\del^\alpha (\chi Fu)}_{L^2(N)} &= \sup_{\norm{u}_{L^2(X_+)}=1}\norm{\langle \del^\alpha H(\cdot), u\rangle}_{L^2(N)} \\& \leq C' \sup_{x\in K} \norm{\del^\alpha_x H(x)}_{L^2(X_+)'} \leq C'CR^{\abs{\alpha}}\alpha!^\sigma\,,
	\end{align*}
	where $C'>0$ is the volume of $N$. With an application of \cref{lem:inA} the proof is complete.
\end{proof}

The above result has immediate consequences for $S_q-S_0$. In the following proofs, for any set $W$, $\mathbb{1}_W$ will both denote the multiplication by the characteristic function of $W$ in $L^2$ as well as the restriction to $L^2(W)$ of some function defined on a larger set.
\begin{lemma}\label{lem:onmani}
	Let $\Sigma\Subset B_{r}(2e_1),\Xi\Subset B_{r}(-2e_1)$ be open, $T'\in(0,T)$ and $\Omega=(0,T')\times \Xi$. There is a closed $(n+1)$-dimensional torus $N\supset \Omega$ (thus an analytic closed manifold), a function $\chi\in G^3(N)\cap C_c^\infty(Y)$, $\rho>0$, and a continuous $b'\colon [0,\infty)\to [0,\infty)$ so that the following is true. 

	For every $q\in L^\infty(U)$ with $\supp \,q\subset \bar\Sigma$, the map
	\[
		F_q\coloneqq \mathbb{1}_{X_+}\chi (S_q-S_0)\circ \chi\mathbb{1}_{X_+}\,,\qquad u\mapsto \mathbb{1}_{X_+}\chi (S_q-S_0)(\chi\mathbb{1}_{X_+}u)%
	\]
	satisfies
	\begin{equation}\label{eq:Fqnormbd}
		F_q \colon L^2(N) \to A^{3,\rho}(N)\quad\text{with}\quad\norm{F_q}_{L^2(N) \to A^{3,\rho}(N)} \leq b'\left(\norm{q}_{L^\infty(U)}\right)\,,
	\end{equation}
	and
	\begin{equation}\label{eq:Fgivessol}
		\mathbb{1}_{\Omega}(S_q - S_0)\circ \mathbb{1}_{\Omega} = \mathbb{1}_{\Omega}F_q\circ \mathbb{1}_{\Omega}\,.
	\end{equation}
\end{lemma}
Before we move on to the proof, we remark that in the definition of $F_q$, we post-compose with $\mathbb{1}_{X_+}$ at the end only so that the adjoint operator of $F_q$ (which can be determined using \cref{lem:adjoint} below) will be well-defined.
\begin{proof}
	Consider the torus $N\subset \RR^{1+n}$ defined by identifying opposite sides of the cube $[-R,R]^{1+n} \subset \RR^{1+n}$, where $R>0$ is chosen sufficiently large so that ${Y}=(-T,T)\times B_{r}(-2e_1)\Subset N$ (see also \cite[Lem.~3.1.8]{zbMATH07625517}). With the Euclidean metric, $N$ is a closed analytic manifold. We take $\chi \in G^3(\RR^{1+n})$ to be identically $1$ in an open neighborhood of $\Omega$ and supported in a compact subset of ${Y}\subset N$ (that Gevrey cut-offs exist follows from \cite[\S~1.4]{doi:10.1142/1550} or \cite[\S~1.4]{hoermander1}). 

	For any $g\in L^2(X_+)$ and $\td{q}\in L^\infty(U)$, $(S_{\td{q}}-S_0)(g)$ is the unique solution of \cref{eq:forward} for $f=-\td{q}S_{\td{q}}(g)$ and $q=0$, whereas $S_0(-\td{q}S_{\td{q}}(g))$ is also a solution of \cref{eq:forward} for the same $f$ and $q$, so that \cref{lem:ex} from the appendix guarantees that $S_0(-\td{q}S_{\td{q}}(g)) = (S_{\td{q}}-S_0)(g)$. In particular, for any $q\in L^\infty(U)$, we have
	\begin{equation}\label{eq:diffgev}
	F_q = \mathbb{1}_{X_+}\chi (S_q - S_0) \circ \chi\mathbb{1}_{X_+} = \mathbb{1}_{X_+}\chi S_0\circ (-qS_q)\circ \chi\mathbb{1}_{X_+}\,,%
	\end{equation}
	and because $\mathbb{1}_{X_+}\chi = 1$ on $\Omega$, this proves \cref{eq:Fgivessol}. 

	From \cref{lem:ex} (see the appendix) and \cref{cor:gev}, we know that $S_0\circ \mathbb{1}_{(0,T)\times\Sigma} \colon L^2(X_+)\to H^1({Y})$ is continuous and maps $S_0\circ \mathbb{1}_{(0,T)\times\Sigma} \colon L^2(X_+)\to G^3({Y})$, so that \cref{prop:kernel} implies that there is $\rho>0$ so that the map
	\[
		\chi S_0\circ \mathbb{1}_{(0,T)\times\Sigma} \colon L^2(X_+) \to A^{3,\rho}(N)\,,\quad f\mapsto \chi S_0(\mathbb{1}_{(0,T)\times\Sigma}f)
	\]
	is continuous, and its operator norm is independent of $q$. Thus, using \cref{eq:diffgev} and noting that $\mathbb{1}_{X_+}\chi S_0 =\chi S_0$ and $\supp \,q\subset \bar\Sigma$, we have 
	\begin{align*}
		\norm{F_q}_{L^2(N)\to A^{3,\rho}(N)} &\leq \norm{\chi S_0\circ \mathbb{1}_{(0,T)\times\Sigma}}_{L^2(X_+)\to A^{3,\rho}(N)} \norm{-qS_q \circ \chi}_{L^2(X_+)\to L^2(X)} \\ &\leq C' \norm{q}_{L^\infty}b\left(\norm{q}_{L^\infty}\right)\,,%
	\end{align*}
	for $b$ from the appendix (Equation~\eqref{eq:Sqbound}) and some $C'>0$ independent of $q$, which gives \cref{eq:Fqnormbd}.
\end{proof}

As it will be required later, we introduce an adjoint operator, see also \cite[Lem.~3.1]{filippas2025stabilityinverseproblemwaves}. %
\begin{lemma}\label{lem:adjoint}
 	Let $q\in L^\infty(U)$ and define
	\[
		S^\ast_q \colon L^2(X_+)\to H^1((0,2T)\times U)\,,\quad S^\ast_q = R\circ S_q\circ R\,,\quad\text{where}\quad R u(t,x) \coloneqq u(T-t,x)\,.
	\]
	The $L^2(X)$-adjoint of $\mathbb{1}_{X_+}S_q\circ \mathbb{1}_{X_+}$ is $\mathbb{1}_{X_+}S_q^\ast\circ \mathbb{1}_{X_+}$.
\end{lemma}
\begin{proof}
	Let us first point out that by construction, for any $g\in L^2(X_+)$, we have $S^\ast_q(g)\in H^1((0,2T)\times U)$, %
	which satisfies
	\begin{equation}\label{eq:propsadj}
		(\Box+q)S^\ast_q(g) = g\quad\text{in}\ \ (0,2T)\times U\,,\quad S^\ast_q(g)\vert_{t>T} = 0\,,\quad S^\ast_q(g)\vert_{(0,2T)\times\del U} = 0\,.
	\end{equation}

	Let $f,g\in L^2(X)$. Using \cref{eq:propsadj}, we have
	\[
		\int_{X} S_q(\mathbb{1}_{X_+} f) \mathbb{1}_{X_+} g \dd(t,x) = %
		\int_{X_+} S_q(\mathbb{1}_{X_+} f) (\Box+q)S^\ast_q (\mathbb{1}_{X_+} g) \dd(t,x)\,,
	\]
	where the RHS can be understood as the pairing between an $H^1$ and an $H^{-1}$ function: %
	by construction, $S_q(f)$ and $S^\ast_q(g)$ vanish in $t<0$ and $t>T$ respectively (and $S_q(f)$ vanishes in $(0,T)\times \del U$), and thus this pairing makes sense.

	Taking a sequence of smooth functions $v_j \in C^\infty(X)$ vanishing on $(0,2T)\times\del U$ and in $t>T$, with $v_j \to S_q^\ast(\mathbb{1}_{X_+}g)$ in $H^1(X)$, the definition of $S_q$ as the solution operator gives (essentially by partial integration, or rather the definition of a weak solution, see \cite[Eq.~(2.67)]{MR1889089}),	%
	\begin{align*}
		&\int_{X_+} S_q(\mathbb{1}_{X_+} f) (\Box+q)S^\ast_q (\mathbb{1}_{X_+} g) \dd(t,x) %
		= \lim_{j\to\infty} \int_{X_+} S_q(\mathbb{1}_{X_+} f) (\Box+q)v_j \dd(t,x)\\
		&\qquad=  \lim_{j\to\infty} \int_{X_+} \mathbb{1}_{X_+} f v_j \dd(t,x) = \int_{X} \mathbb{1}_{X_+} f S^\ast_q (\mathbb{1}_{X_+} g) \dd(t,x)\,,
	\end{align*}
	which completes the proof.
\end{proof}

Finally, we shall introduce one additional piece of notation from \cite[Thm.~3.16]{zbMATH07465842}. For $\sigma \in [1,\infty),\rho > 0$ and $s\in \RR$, and some closed smooth manifold $M$ we let 
\begin{equation*}%
	W^{\sigma,\rho}(H^s, H^{-s}) \coloneqq \{T\in B(H^s,H^{-s})\colon T(H^s)\subset A^{\sigma,\rho}(M)\,,\ \ T^\ast(H^s)\subset A^{\sigma,\rho}(M)\}\,,
\end{equation*}
where we wrote $H^s$ for $H^s(M)$ and $B(H^{s},H^{-s})$ is the set of bounded linear operators between $H^s(M)$ and $H^{-s}(M)$, and for any $T\in B(H^s,H^{-s})$, $T^\ast \in B(H^s, H^{-s})$ denotes the formal adjoint of $T$. It is shown in \cite[Thm.~3.16]{zbMATH07465842} that $W^{\sigma,\rho}(H^s,H^{-s})$ is a Banach space with the norm
\[
	\norm{T}_{W^{\sigma,\rho}(H^s, H^{-s})} \coloneqq \max\left\{\norm{T}_{H^s\to A^{\sigma,\rho}},\norm{T^\ast}_{H^s\to A^{\sigma,\rho}}\right\}\,.
\]

We now have all tools in hand to provide the
\begin{proof}[Proof of \cref{thm:main}]
	We set up notation to apply \cite[Thm.~4.2(b)]{zbMATH07465842} (or rather the analogous statement thereof when one uses \cite[Prop.~3.11]{zbMATH07465842} in its proof to allow restricting the support of elements in $K$ as we have done here). %
	For $N,F_q$ from \cref{lem:onmani}, define the operator
	\begin{equation*}%
		F\colon K \to B(L^2(N),L^2(N))\,, \quad q\mapsto F_q\,,%
	\end{equation*}
	where $B(L^2(N),L^2(N))$ denotes the set of bounded linear operators between $L^2(N)$ and $L^2(N)$. 

	As a consequence of \cref{lem:onmani} (in particular \cref{eq:Fqnormbd}) and \cref{lem:adjoint}, $F$ maps $K$ into a bounded set in $W^{3,\rho}(H^0,H^0)$ for some $\rho>0$.
	By an application of \cite[Thm.~4.2(b)]{zbMATH07465842} we conclude that if $\omega$ is a modulus of continuity so that	
	\begin{equation}\label{eq:bound}
		\norm{q_1-q_2}_{H^\mu(U)} \leq \omega\left(\norm{F(q_1)-F(q_2)}_{L^2(N)\to L^2(N)}\right)\,,\qquad q_1,q_2\in K\,,
	\end{equation}
	then we must have $\omega(s) \gtrsim \abs{\log s}^{-\delta\frac{6n+7}{n}}$ for $s$ small.

	Due to \cref{eq:Fgivessol}, $\norm{\mathbb{1}_\Omega(S_{q_1}-S_{q_2})\circ\mathbb{1}_{\Omega}}_{L^2(\Omega)\to L^2(\Omega)} \leq \norm{F(q_1)-F(q_2)}_{L^2(N)\to L^2(N)}$, so that \cref{eq:bound} completes the proof.
\end{proof}

\appendix 

\subsection{Existence of Solutions}

\begin{lemma}\label{lem:ex}
	There is a continuous function $b\colon [0,\infty)\to [0,\infty)$ so that for every $q\in L^\infty(U)$ there is a continuous linear operator
	\[
		S_q \colon L^2(X_+) \to %
		C([-T,T];H^1_0(U))\cap C^1([-T,T];L^2(U)) \subset H^1(X)\,, %
	\]
	with 
	\begin{equation}\label{eq:Sqbound}
		 \norm{S_q}_{L^2(X_+)\to H^1(X)} \leq b\left(\norm{q}_{L^\infty(U)}\right)\,,
	\end{equation}
	so that in $C([-T,T];H^1_0(U))\cap C^1([-T,T];L^2(U))$ the unique solution to \cref{eq:forward} is $u = S_q(f)$.%
\end{lemma}
\begin{proof}%
	Translating the condition $u\vert_{t<0}=0$ in \cref{eq:forward} into vanishing Cauchy data at $t=0$, the uniqueness and existence of the solution $u=u_q^f$ as well as the continuity of $S_q$ are guaranteed by \cite[Thm.~2.30]{MR1889089} (via extension by $0$ to $t<0$ of the solution constructed there). %
	
	We turn to finding the explicit norm bound on $S_q$, assuming at first that $q\in C^1(\bar U)$. One can directly verify that the solution $u$ constructed above is also a weak solution according to the definition given in \cite[\S~7.2]{MR2597943}, and by inspection of the proof of \cite[\S~7.2,~Thm.~2]{MR2597943} (see also \cite[\S~7.2,~Thm.~5(i)]{MR2597943}), one finds the bound in \cref{eq:Sqbound} for $q\in C^1(\bar U)$.
	On the other hand, if $q\in L^\infty(U)$ and we let $(q_k)_{k\in\mathbb{N}} \subset C^1(\bar U)$ converge to $q$ in $L^\infty$, by direct verification of the definition of a weak solution (see \cite[Eq.~(2.66)]{MR1889089}), we will see that $S_{q_k} \to S_q$ as operators $L^2(X_+)\to H^1(X)$, completing the proof.
\end{proof}

\bibliographystyle{amsplain}
 
\end{document}

%% file: mydefs.tex
\usepackage{bm}
\usepackage{mathrsfs}

\usepackage{tikz}
\usetikzlibrary{arrows}
\usetikzlibrary{positioning}
\usetikzlibrary{shapes.geometric}
\usepackage{pgf}
\usetikzlibrary{patterns}
\usetikzlibrary{patterns.meta}
\usepgflibrary{plothandlers} % LaTeX and plain TeX and pure pgf
\usepgflibrary[plothandlers] % ConTeXt and pure pgf
\usetikzlibrary{plothandlers} % LaTeX and plain TeX when using TikZ
\usetikzlibrary[plothandlers] % ConTeXt when using TikZ
\usetikzlibrary{topaths} % LaTeX and plain TeX
\usetikzlibrary[topaths] % ConTeXt
\usetikzlibrary{decorations.markings}

\makeatletter
\newcommand{\leqnomode}{\tagsleft@true\let\veqno\@@leqno}
\newcommand{\reqnomode}{\tagsleft@false\let\veqno\@@eqno}
\makeatother
\newcommand{\eps}{\varepsilon}
\newcommand{\sem}{\setminus\{0\}}
\newcommand{\RR}{\mathbb{R}}

\newcommand{\dd}{\mathrm{d}}

\newcommand{\norm}[1]{\left\lVert#1\right\rVert}
\newcommand{\abs}[1]{\left\lvert#1\right\rvert}
\newcommand{\del}{\partial}
\newcommand{\supp}{\mathrm{supp}}

\newcommand{\td}[1]{\tilde{#1}}

\newcommand{\bdry}{(h\td{\del}_\nu+(\td{\del}_\nu\phi))}

\usepackage{xparse}

\NewDocumentCommand{\op}{O{h}m}{\mathrm{Op}_{#1}\left(#2\right)}
\NewDocumentCommand{\tbdry}{O{}}{\tau_{#1}\bdry}